\documentclass[12pt]{amsart}
\usepackage{enumerate}
\usepackage{amssymb}
\usepackage{amsmath}
\usepackage{amsthm}
\newcommand{\snug}{\unskip\kern-\mathsurround}
\newcommand{\ad}{{\rm ad}}

\newcommand{\gr}{{\rm gr}}

\newcommand{\cd}{{\rm cd}}
\newcommand{\Z}{{\mathbb Z}}
\newcommand{\F}{{\mathbb F}}

\newcommand{\Q}{{\mathbb Q}}

\newtheorem{theorem}{Theorem}

\newtheorem{corollary}[theorem]{Corollary}

\theoremstyle{definition}

\newtheorem{example}[theorem]{Example}
\begin{document}
\title{\bf The Genesis of a Theorem} \author{John Labute}
\address{Department of Mathematics and Statistics, McGill University, Burnside Hall, 805 Sherbrooke Street West, Montreal QC H3A 2K6, Canada}
\email{labute@math.mcgill.ca}
\thanks{This paper was written at the El Monte Eco Lodge in Mindo, Ecuador. The author wishes to thank Mariela, Tom and their staff and guides for providing a stimulating and restful environment}
\begin{abstract}
In this article we trace the genesis of a theorem that gives for the first time examples of Galois group $G_S$  of the maximal $p$-extension of $\Q$, unramified outside a finite set of primes not containing $p$, that are of cohomlogical dimension $2$. The pro-$p$-group $G_S$ is a fab pro-$p$-group which means that all its derived factors are finite.  
\end{abstract}
\date {June 23, 2024}
\maketitle
\hfill {\it \`{A} Serre\ \ \ \ \ \ }\section{The Theorem}\label{A}
Let $p$ be a odd prime and let $S=\{q_1,\ldots,q_m\}$ be a set of primes $q_i\equiv 1$ mod $p$. Let  $g_i$ be a primitive root mod $q_i$ and let the {\bf linking number} $\ell_{ij}\in\F_p$ be defined by
$$
q_i\equiv g_j^{-\ell_{ij}}\ {\rm(mod}\ q_j).
$$
If $g$ is another primitive root mod $q_j$ then $\ell_{ij}$ is replaced by $c_j\ell_{ij}$ for some $c_j$.
Let $G_S$ be the Galois group of the maximal $p$-extension of $\Q$ which is unramified outside of the set $S$. Not much was known about these groups; all we knew was that, by Golod-Shafarevich, they were infinite if $m\ge4$.
\begin{theorem}
Suppose that $m$ is even and that 
\begin{enumerate}[\rm (a)]
\item $\ell_{ij}=0$ if $i,j$ are both odd,
\item $\ell_{12}\ell_{23}\cdots\ell_{m-1,m}\ell_{m1}-\ell_{21}\ell_{32}\cdots\ell_{m,m-1}\ell_{1m}\ne0$.
\end{enumerate}
Then $G_S$ is of cohomological dimension 2.
\end{theorem}
\begin{example}
For $S=\{7,19,61,163\}$ and $p=3$ the non-zero $\ell_{ij}$ are
$$
\ell_{12}=\ell_{21}=\ell_{14}=\ell_{23}=\ell{_{24}}=\ell_{34}=1, \  \ell_{43}=\ell_{41}=-1.
$$
\end{example}
Conditions (a) and (b) are satisfied, so $\cd(G_S)=2$ and this gives the first example of a fab pro-$p$-group whose cohomological dimension is $2$. 
The most general statement of our criterion which covers the case $m$ is odd is elegantly formulated by Alexander Schmidt:
\begin{theorem}\label{B}
Let $G$ be a finitely generated pro-$p$-group such that $H^1(G,\Z/p\Z)$ is the direct sum of non-trivial subspaces $U,V$ and such that the cup product
$$ H^1(G,\Z/p\Z)\otimes H^1(G,\Z/p\Z)\rightarrow H^2(G,\Z/p\Z)$$
\begin{enumerate}[\ \rm(a)]
\item is trivial on $U\otimes U$ and 
\item maps $U\otimes V$ surjectively onto $H^2(G,\Z/p\Z)$,
\end{enumerate}
then $G$ is of cohomological dimension $2$.
\end{theorem}

\section{It all began with a question of Serre}
In a 1963 S\'eminaire Bourbaki Lecture ~\cite{Serre63} Serre proposed the following question where $F$ is a finitely generated free pro-$p$-group and $F_2=F^p[F,F]$.
:

 "Soit $r\in F_2$, et soit $G_r=F/(r)$. Peut-on \'etendre \`a $G_r$ les r\'esultats d\'emontr\'es par Lyndon~\cite{Lyndon} dans le cas discret? En particulier, si $r$ n'est pas une puissance $p$-i\`eme, est-il vrai que $G_r$ est de dimension cohomologique $2$?"
 \par
 In the Fall of 1964 Serre introduced me to Lie algebras and the work of Lazard on  filtrations of groups in his course "Groups and Lie Algebras" which he gave at Harvard. He also gave me a private lecture on the Elimination Theorem for groups and Lie algebras which as we shall see played a decisive role in the proof of Theorem 1. My work is very much influenced by the interplay between groups and Lie algebras. Sometimes, if one can prove the corresponding Lie algebra question and the result is strong enough, it can be pulled back to a proof of the original group theory question. This is the case here, at least partially so; so let me give a sketch of my attempt to solve his question with this in mind. I was so excited with my "proof" that I told Serre I had answered his question in the affirmative. It was only when I wrote down all the details of my attempted proof that I realized that there was a natural boundary inherent in my method. I was embarassed having to tell Serre this especially when he said he told Tate that I had solved his question.
\par
 
 Let $R= (r)$ and let $M=R/[R,R]$ where $[R,R]$ is the subgroup of $R$ generated by the commutators $[x,y]=x^{-1}y^{-1}xy$ with 
 $x,y \in R$. Then by a result of Brumer~\cite{Bru66} we have $\cd(G_r)=2$ if and only if $M$, viewed as a module over the completed algebra $\Z_p[[G]]$ of $G$, is a free module of rank $1$  generated by the image of $r$ in $M$.  In ~\cite{Lab67} we showed that this was true if $r$ was not "too close" to a $p$-th power. 
 More precisely, let $(F_n)$ be the filtration of $F$ defined by $F_1=F$ and $F_{n+1}=F^p[F,F_n]$, also known as the descending $p$-central series of $F$, and let $e$ be largest with $r\in F_e$; we have $e<\infty$ if $r\ne1$, since the intersection of the subgroups $F_n$ is $1$. We show that $\cd(G_r)=2$ if $r$ is not a $p$-th power mod $F_{e+1}$.
   \par
   
  The filtration $(F_n)$ has two important properties: 
  \begin{enumerate}[\ \rm (a)]
 \item $[F_n,F_m]\subset F_{n+m}$, 
 \item $F_n^p\subset F_{n+1}$.
 \end{enumerate}
 
  Let $\gr_n(G)$ be the abelian group $G_n/G_{n+1}$, which we denote additively. We introduce a Lie bracket on 
  $\gr(G)=\oplus\gr_n(G)$ using the commutator operation as follows: let $\xi_n \in \gr_n(G), \eta_m\in\gr_m(G)$ be the the images of $x_n\in
 G_n, y_m\in G_m$ respectiively. Then, letting $[\xi_n,\eta_m]$ be the image in $\gr_{n+m}$ of $[x_n, y_m]\in G_{n+m}$, we get a Lie bracket on the graded $\F_p$-module $\gr(G)$. The $p$-th power operator on $G$ induces the structure of a graded 
 $\F_p[\pi]$-algebra on $\gr(G)$; namely, if $x\in F_n$ and $\xi$ is its image in $\gr_n(G)$ then $\pi\xi$ is the image of $x^p$ in $\gr_{n+1}(G)$. If $F=F(x_1,\ldots,x_d)$ is the free group on $x_1,\ldots,x_d$ and $\xi_i$ is the image of $x_i$ in $\gr_1(F)$ then $L=\gr(F)$ is the free Lie algebra over $\F_p[\pi]$ on $\xi_1,\ldots,\xi_d$.
 
 Let $\rho$ be the image of $r$ in $\gr_e(F)$; this element is called the {\bf initial form} of $r$.  Then $\rho$  not a multiple of $\pi$ is the same as saying $r$ is not a $p$-th power modulo $F_{e+1}$.   Let 
 $\mathfrak r$ be the ideal of $L$ generated by $\rho$ and let $\mathfrak m=\mathfrak r/[\mathfrak r,\mathfrak r]$. Let  $\mathfrak g= L/\mathfrak r$ and let $U_\mathfrak g$ be the enveloping algebra of $\mathfrak g$. Then $\mathfrak m$ is a $U_\mathfrak g$ module via the adjoint representation. Then the corresponding Lie algebra question would be to show that $\mathfrak r/[\mathfrak r,\mathfrak r]$ was a free $U_\mathfrak g$ module of rank $1$ generated by the image of $\rho$. To prove this we needed to prove that $\mathfrak g$ was torsion free as an $\F_p[\pi]$ module which would entail the same is true for $U_\mathfrak g$ and hence by the Birkhoff-Witt Theorem that $U_\mathfrak g$ has no zero divisors. It would be then a straightforward exercise using the exact sequence
 $$
 \mathfrak r/[\mathfrak r,\mathfrak r]\rightarrow I_\mathfrak g\rightarrow U_g\rightarrow \F_p[\pi]\rightarrow 0
 $$
 where $I_\mathfrak g\cong U_{\mathfrak g}^d$ is the augmentation ideal of $U_\mathfrak g$. This yields the resolution by free $U_\mathfrak g$ modules
 $$
0\rightarrow \mathfrak r/[\mathfrak r,\mathfrak r]\rightarrow I_\mathfrak g \rightarrow U_g\rightarrow \F_p[\pi]\rightarrow 0
 $$
 showing that $\cd(\mathfrak g)=2$.
 \par
 To show that this lifts to $G=F/R$ we have to show that $\gr (G)=L/(\rho)$ or, equivalently that $\gr(R)=(\rho)$ where $R_n=F\cap F_n$.
 The proof of this requires that $\mathfrak g$ be torsion free and utilizes suitable Lazard filtrations of $F$ and iterative applications of Birkhoff-Witt; it is much too lengthy to be even sketched here. The same goes for the proof of $L/(\rho)$ being torsion free if $\rho$ is not a multiple of $\pi$. Today I marvel that I could come up with such intricate proofs but this is largely due to the influence of Serre at Harvard in 1964 and in Paris in 1965/67.
\par

This proof generalizes to the case of several relators; cf.~\cite{Lab85} but the linear independence of the initial forms of the relators over $\F_p[\pi]$ is not enough to prove the result. We need to assume that $U_\mathfrak g$ is a torsion free $\F_p[\pi]$ module and that $\mathfrak r/[\mathfrak r,\mathfrak r]$ is a free $U_{\mathfrak g}$ module on the images of the initial forms of the relators. We call such relators {\bf strongly free}.

 As it turned out, the question of Serre had a negative answer; for example, in the case  $r=x^p[x^p,y]$ as was shown by Guildenhuys~\cite{Gild68} since $x^p=1$ in $G_r$ and $r$ is not a $p$-th power in $F$. We still do not have a criterion for deciding whether the group $G_r$ has finite or infinite cohomological dimension. 
 
 \section{A fortitutitous detour to discrete groups}
 In this section $F$ will be a free discrete group of rank $m$. We let $(F_n)$ be the filtration of $F$ defined by
 $F_1=F, F_{n+1}=[F,F_n]$, also known as the descending central series of $F$. The graded Lie algebra $\gr(F)$ is defined as above. It is a free Lie algebra over $\Z$. Let $r\in F$ and let $G=F/(r)$. Suppose $r\in F_e$, $r\notin F_{e+1}$ with $e\ge 2$. Let $\rho$ be the image of $r$ in $\gr_e(F)$, the initial form of $r$. If $\rho$ is not a proper multiple then Waldinger in~\cite{Wald67} showed that 
 $\gr_n(G)$ is a free $\Z$ module for $e\le n\le 3e$ and gave formulae for the ranks as a partial answer to a question of Magnus who asked if $\gr(G)$ was torsion free for $G=F/(r_1,\ldots,\_d)$ if the initial forms of the $r_i$ were linearly independent.
 
 In \cite{Lab70} it is shown that in fact $\gr(G)=\gr(F)/(\rho)$ is a free $\Z$ module if $\rho$ is not a proper multiple and that the Poincar\'e series of its enveloping algebra was 
 $$
 \frac{1}{1-mt+t^e}=\prod_{n\ge1}\frac{1}{(1-t^n)^{g_n}}
 $$
 where $g_n$ is the rank of $\gr_n(G)$. A straightforward calculation yields the formula
 $$
 ng_n = \sum_{d|n}\mu( n/d)[\sum_{0\le i\le n/d}(-1)^i \frac{1}{d+i-ei}\binom{d+i-ei}{i} m^{d-ei}  ] .
 $$
In particular $g_n$ depends only on $n, e$ and $m$.

In \cite{Lab85} we extended these results to the case of several relators provided that their initial forms $\rho_1,\ldots,\rho_d$
 satisfied the following two conditions
\begin{enumerate}[\rm (a)]
\item The enveloping algebra $U$ of $\gr(F)/(\rho_1,\ldots,\rho_d)$ is a torsion free $\Z$ module;
\item If $\mathfrak r=(\rho_1,\ldots,\rho_d)$ then $\mathfrak r/[\mathfrak r,\mathfrak r]$ is a free $U$ module with basis the the images of $\rho_1,\ldots,\rho_d$.
\end{enumerate}
Such a sequence we also call strongly free. 
We also gave a method for constructing such sequences using  using the Elimination Theorem. This method would prove decisive in the proof of Theorem 1.

\begin{theorem}[Elimination Theorem]
Let $K$ be a commutative ring and let $L=L(X)$ be the free Lie algebra over $K$ on the set $X$. Let $S$ be a subset of $X$ and let $\mathfrak s$ be the ideal of $L(X)$ generated by $X-S$. Let $W$ be the enveloping algebra of $L(S)$ and let $M(S)$ be the submonoid of $W$ generated by $S$.Then $\mathfrak s$ is the free Lie algebra over $k$ on the elements
$$
T=\{\ad(m)(x)\ | \ x\in X-S,\  m\in M(S)\}.
$$
\end{theorem}
\begin{corollary}
The $W$ module $\mathfrak s/[\mathfrak s,\mathfrak s]$  is a free module over $W$ with basis the image of $T$.
\end{corollary}

\begin{theorem}[Criterion for Strong Freeness]
Let $\rho_1,\ldots,\rho_d$ be elements of $\mathfrak s$ such that the elements
$$
T_1=\{\ad(m)(\rho_j)| m\in M(S),\  0\le j\le d\}
$$
are part of a basis of $\mathfrak s$. Then $\rho_1,\ldots \rho_n$ is a strongly free sequence.
\end{theorem}

\begin{proof}
The ideal $\mathfrak r$ of $L$ generated by $\rho_1,\ldots,\rho_d$ is generated as an ideal of $\mathfrak s$ by the set $T_1$.
Since $T_1$ is part of a basis of $\mathfrak s$ the Corollary to the Elimination Theorem says that $\mathfrak r/[\mathfrak r,\mathfrak r]$ is a free module over the enveloping algebra $V$ of $\mathfrak s/\mathfrak r$ with basis the image of $T_1$.
The exact sequence
$$
0\rightarrow \mathfrak s/\mathfrak r \rightarrow U/\mathfrak r\rightarrow L(S)\rightarrow 0
$$
splits which implies that 
$$
U=V\otimes W=\bigoplus_{m\in M(S)}Vm.
$$
This implies that $U$ is a free $K$ module. To show that $\mathfrak r/[\mathfrak r,\mathfrak r]$ is  a free modulue let $\overline\rho_i$ be the image of $\rho_i$ in $\mathfrak r/[\mathfrak r,\mathfrak r]$ and suppose that
$$
\sum_{i} u_i\rho_i=0 \ with\  u_i\in U.
$$
Then $u_i=\sum_{j}v_{ij}m_j$ with $u_i\in U$, $v_{ij}\in V$ which implies that
$$
\sum_i u_i\overline\rho_i=\sum_{i,j}v_{ij}\ad(m_j)(\overline\rho_j)=0.
$$
But this implies $v_{ij}=0$ since $\mathfrak r/[\mathfrak r,\mathfrak r]$ is a free $V$ module with basis $T_1$.

\end{proof}
The Lie algebras $L/\mathfrak r$ constructed in this way are of cohomological dimension $\le 2$.

An important example of this is
$$
[\xi_1,\xi_2], [\xi_2,\xi_3], [\xi_3,\xi_4], \dots,[\xi_{m-1},\xi_m],
$$
 where $X=\{x_1,\ldots,x_m\}$ and 
$S=\{x_1,x_3,x_5,\dots\}$ so $X-T=\{x_2,x_4, x_6,\ldots\}$.
\medskip
\par
In \cite{Anick87}, Anick gave the name {\bf mild group}  to a discrete group with defining relators whose initial forms are strongly free and gave many examples using his criterion of {\bf combinatorial freeness} to prove mildness. The most important of these is the fundamental group $G$ of the complement of a {\bf pure braid link} $\mathcal L$ in $S^3$ which is obtained from a pure braid by identifying the top and bottom of each strand. If $m$ is the number of strands of the braid then $G$ has the presentation $G=F/(r_1,\ldots,r_{m-1})$, where $r_i=[x_i,y_i]$ with
$$
y_i=x_i^{a_i}\prod_{j=1}^{m}x_j^{a_{ij}}\ {\rm mod}\ [F,F].
$$
so that 
$$
r_i=\prod_{j=1}^{n}[x_i,x_j]^{a_{ij}} \ {\rm mod}\  [F,[F,F]].
$$
The matrix $(a_{ij})$ is a symmetric matrix with zero diagonal because $a_{ij}$ is the {\bf linking number} between the $i$-th and $j$-th unknot of the link. We assume that the matrix $a_{ij}$ has no zero rows so that the image $\rho_i$ in $\gr_2(F)$ is non-zero. We then have
$$
\rho_i=\sum_{j=1}^{m}a_{ij}[\xi_i,\xi_j],
$$
where $\xi_i$ is the image of $x_i$ in $\gr_2(F)$. Anick uses a weighted graph associated to the matrix $(a_{ij})$ to give a criterion for the  strong freeness of the sequence $\rho_1,\ldots,\rho_{m-1}$; note that $\sum\rho_i=0$. This graph, which Anick calls a {\rm linking diagram},  has as vertices the set $\{\xi_1,\ldots,\xi_m\}$ with $\xi_i,\xi_j$ being joined if $a_{ij}\ne0$; in this case, the weight being $a_{ij}$. The graph is connected mod $p$ if and only if there is a spanning subtree whose vertices are not congruent to zero modulo $p$ where $p$ can be any prime.
\par
Anick then shows that if the linking diagram of the matrix $(a_{ij})$, or of the link $\mathcal L$, is connected mod $p$ then the sequence 
$\rho_1,\ldots,\rho_i$ is strongly free mod $p$ and strongly free if it is connected mod $p$ for any prime $p$.

\section{Back to pro-$p$-groups}
In ~\cite{Koch77} Koch uses Lazard filtrations of the completed group algebra $\F_p[[F]]$ to show that if $G=F(x_1,\ldots,x_m)/(r_1,\ldots,r_d)$ we have $\cd(G)=2$ if the initial forms of the relators form a strongly free sequence of Lie elements in the free Lie subalgebra $\mathfrak L$ of $\gr(F)$ on $\xi_,\ldots,\xi_m$, the initial forms of $x_1,\ldots,x_m$. The $\F_p$-algebra $\mathfrak A=\gr(F)$ is the free associative algebra over $\F_p$ on the elements $\xi_1,\ldots,\xi_m$. For Koch, strong freeness means that, if $\mathfrak R$ is the ideal of $\mathfrak A$ generated by $\rho_1,\ldots,\rho_d$ and $\mathfrak I$ is the augmentation ideal of $\mathfrak A$, then $\mathfrak A/\mathfrak I\mathfrak R$ is a free $\mathfrak A/\mathfrak R$ module on the images of the $\rho_i$. But $\mathfrak B=\mathfrak A/\mathfrak R$ is the enveloping algebra of $\mathfrak g=\mathfrak L/\mathfrak r$, where $\mathfrak r$ is the ideal of $\mathfrak L$ generated by the $\rho_i$ and $\mathfrak R/\mathfrak I\mathfrak R$ is isomorphic to $\mathfrak r/[\mathfrak r,\mathfrak r]$, the isomorphism being induced by the inclusion $\mathfrak r\subset\mathfrak R$. So his definition was the same as ours but this connection was not made in~\cite{Koch77}.
\par
Koch also gives a criterion for strong freeness. To describe it, let
$$
\rho_i=\sum_{j=1}^{d} a_{ij}\xi_j
$$

\noindent where $a_{ij}=\partial_j\rho_i\in\mathfrak A$ (Fox derivative). The free commutative, associative $\F_p$-algebra $\tilde{\mathfrak A}$ on $\xi_1,\ldots,\xi_m$ is naturally a quotient of $\mathfrak A$; we identify $\xi$ with its image in $\tilde{\mathfrak A}$. {\bf Koch's Criterion} is that the rank of the ${d\times m}$ matrix $M_K=(\tilde a_{ij})$ be equal to $d$.
If Koch's Criterion holds, one has $d<m$ because of the relations
$$\sum_{j=1}^{j=d} \partial_j\rho_i\xi_j=0,\ 1\le i\le m$$
which shows that the columns of $M_K$ are linearly dependent.

For the relators $\rho_1=[\xi_1,\xi_2], \rho_2=[\xi_2,\xi_3],\rho_3=[\xi_{3},\xi_4]$ where $m=4$ we have
$$
M_K=\begin{bmatrix}-\xi_2&\xi_1&0&0\\
				0&-\xi_3&\xi_2&0\\
				0&0&-\xi_4&\xi_3\end{bmatrix}
$$
which is of rank $3$ and so the relators are strongly free.
\par 
In~\cite{Koch70}, using local and global classfield theory,  Koch gives a presentation for the Galois group $G$ of the maximal $p$-extension of $\Q$ which is unramified outside a finite set $S$ of primes which is analogous to that of the fundamental group of the complement of a  tame link in $S_3$. 
If $S=\{q_1,\ldots,q_m\}$ with $q_i\equiv 1$ mod $p$, he shows that $G=F(x_1,\ldots,x_m)/(r_1,\dots, r_m$ where
$$
r_i=x_i^{q_i-1}\prod_{j\ne i}[x_i,x_j]^{\ell_{ij}} \ {\rm mod}\  F_3
$$
where  for $j\le m$,  $\ell_{ij}$ is the image in $\Z/pZ$ is any integer $r$ with
$$
q_i\equiv g_j^{-r} \ {\rm mod\  q_j}
$$
where $g_i$ is a primitive root mod $q_i$. 
\par 
If $S=\{q_1,\ldots, q_m, p\}$ the presentation is the same except for one additional variable $x_{m+1}$ but with $\ell_{ij}$ defined as before for $i,j\le m$ while for $i\le m, j=m+1$, $\ell{ij}$ is defined by
$$
q_i\equiv(1+p)^{-\ell_{ij}}\ {\rm mod} \ p.
$$
In both cases there is a linking diagram associated to the matrix $(\ell_{ij})$.
There is a striking similarity between the latter presentation and that of the complement of a pure braid link in $S^3$. In the Galois case we have $\cd(G)=2$, cf.~\cite{Bru66}, while we can't always prove this for the link group. 
\par
The first presentation is strikingly different since $\ell_{ij}\ne\ell_{ji}$ in general, the group has the same number of generators as relators and $G$ is fab.  If $\rho_i$ be the image of $r_i$ in $\gr_2(F)$ we have
$$
\rho_i=c_i\pi\xi_i+\sum_{j\ne i}\ell_{ij}[\xi,\xi_j]
$$
where $c_i=(q_i-1)/p$. Reducing mod $\pi$ we get the elements
$$
\overline\rho_i=\sum_{j\ne i}\ell_{ij}[\xi,\xi_j]
$$
in the free lie algebra over $\F_p$ on the $\xi_i$,  where $\xi_i$ is the image of $x_i$ in $\gr_1(F)$. The relators $(\rho_i)$ are strongly free iff the relators $(\overline\rho_i)$ are strongly free.

At this point in time there was not even one example of a pro-$p$-group $G$ with $G/[G,G]$ finite and $\cd(G)=2$. Even worse the Galois group in question has all of its derived factors finite because by class field theory, the maximal abelian $p$-extension of a number field which is unramified outside a finite set of primes not divisible by $p$ is of finite degree since the ramification is tame for such primes. Such a group is called is a {\bf fab group}. In this case the sentiment was that such a group must have torsion, so could not have finite cohomological dimension.

\section{The lightening bolt hits}
While on leave at Western University in London, Ontario in the Fall of 2004, discussing with Jan Min\'a\u c criteria for strong freeness of the relators in the Koch presentation for $G_S$ with $p\in S$, it suddenly dawned on me while reviewing the criteria for strong freeness in my paper~\cite{Lab85}  that, if $m$ was even and $\ge4$, I could prove that the relators
$$
\rho_1=[\xi_1,\xi_2],\rho_2=[\xi_2,\xi_3]\ldots \rho_{m-1}=[\xi_{m-1},\xi_m],\rho_m=[\xi_m,\xi_1]
$$
\noindent
are strongly free by the Elimination Theorem with $K=\F_p$ and $S=\{\xi_i| i\ {\rm odd}\}$. Note in this case, the linking diagram is a circuit. This meant that I was able to prove that the relators
$$
x_1^p[x_1,x_2],x_2^p[x_2,x_3],\ldots ,x_{m-1}^p[x_{m-1},x_m],x_m^p[x_m,x_1]
$$
are also strongly free giving the first example of a pro-$p$-group $G$ of cohomological dimension $2$ with $G/[G,G]$ finite.

Motivated by this I started to search for circuits in the linking diagram of the Galois group of the maximal $p$-extension unramified outside $S=\{q_1,\ldots, q_m\}$ with $m=4$, $p=3$. For $q_1=7$, $q_2=19$, $q_3=61$, $q_4=163$ I found the circuit 
$$
[\xi_1,\xi_2], [\xi_2,\xi_3], [\xi_3,\xi_4], [\xi_4,\xi_1]
$$
in the linking diagram of its presentation (mod $\pi$)
\begin{align*}
\overline\rho_1&=[\xi_1,\xi_2]+[\xi_1,\xi_4]\\
\overline\rho_2&=[\xi_2,\xi_1]+[\xi_2,\xi_3]+[\xi_2,\xi_4]\\
\overline\rho_3&=[\xi_3,\xi_4]\\
\overline\rho_4&=-[\xi_4,\xi_1]-[\xi_4,\xi_3].
\end{align*}
To prove that theses relators are strongly free it is enough, by our criterion for strong freeness, to prove that they are part of a basis of $\mathfrak s=(\xi_2,\xi_4)$. But because our Lie agebras are graded over $\F_p$, it is enough to prove the linear independence of their images $\overline\rho_i $ in $\mathfrak s/[\mathfrak s,\mathfrak s]$ are linearly independent. Since the $\overline\rho_i$ lie in the subspace spanned by the images of
$$
[\xi_1,\xi_2], [\xi_2,\xi_3], [\xi_3,\xi_4], [\xi_4,\xi_1].
$$
 we just have to prove that the rank of the matrix of the $\overline\rho_i$ with respect to this basis is $4$. But this matrix is
$$
\begin{bmatrix}1&0&0&-1\\
		      -1&1&0&0\\
		      0&0&1&0\\
		      0&0&-1&1\\
\end{bmatrix}
$$
whose determinant is $1$ and we have won! The proof in the general case of Theorem 1 is quite similar.
\begin{corollary}\label{C}
Let $A,B$ be a partition of $\{1,\ldots,d\}$. Let $\mathfrak T$ consist of the elements $[\xi_i,\xi_j]=\ad(\xi_i)(\xi_j)$ with $i\in A$, $j\in B$. Let $\rho_1,\ldots,\rho_m$ be the elements of $\mathfrak s$ with
$$\rho_k=\sum_{1\le i<j\le m}\overline a_{ijk}[\xi_i,\xi_j].$$
If they satisfy conditions 
\begin{enumerate}[\rm \ (a)]
 \item $\overline a_{ijk}=0$ when $i,j\in A$, 
  \item the rank of $M_L$ is $m$,
  \end{enumerate}
 the sequence $\rho_1,\ldots,\rho_m$ is strongly free in $L$.
\end{corollary}
Condition (a) implies that the $\rho_k$ lie in $\mathfrak T$ modulo $[\mathfrak s,\mathfrak s]$ and condition (b) implies the linear independence of the sequence $\rho_1,\ldots,\rho_m$ modulo $[\mathfrak s,\mathfrak s]$. 

Going back to Theorem 1, the Koch presentation of $G_S$ has $d=m$ and the relators are
$$r_i=x_i^{q_i-1}\prod_{j\ne i}[x_i,x_j]^{\ell_{ij}}w_i$$
with $w_i$ in the third term of the descending $p$-central series of $F$. This implies that the holomony relators are
$$\rho_i=\sum_{j\ne i}\ell_{ij}[\xi_i,\xi_j].$$
If the conditions of Theorem 1 are satisfied, that $\rho_1,\ldots,\rho_m$ is a strongly free sequence follows from Corollary 7 if we take $S$ to be the $\xi_i$ with $i$ even. Indeed, if we index the columns of the the matrix $M_L$ by $(1,2), (2,3),\ldots,(m-1,m),(1,m)$ we have
$$
M_L=\begin{bmatrix} \ell_{12}&0&0&\cdots&0&\ell_{1m}\\
                -\ell_{21}&\ell_{23}&0&\cdots&0&0\\
                0&-\ell_{32}&\ell_{34}&\cdots&0&0\\
                0&0&-\ell_{43}&\cdots&0&0\\
                \vdots&\vdots&\vdots&&\vdots&\vdots\\
                0&0&0&\cdots&\ell_{m,m-1}&0\\
                0&0&0&\cdots&-\ell_{m,m-1}&-\ell_{m1}\end{bmatrix},$$
 whose determinant is $\ell_{12}\ell_{23}\cdots\ell_{m-1,m}\ell_{m1}-\ell_{21}\ell_{32}\cdots\ell_{m,m-1}\ell_{1m}$.
 
 So the conditions of Theorems 1 are conditions for the strong freeness of the holonomy relators of a given presentation of a given pro-$p$-group.

\section{Uravelling the Statement of Theorem~\ref{B}}
 Let $G$ be a finitely generated pro-p-group. The cohomology group $H^i(G,\Z/p\Z)$ will be denoted by $H^i(G)$; it is a vector space over the finite field  $\F_p$. The pro-p-group $G$ is said to be of cohomological dimension $n$ if $H^n(G)\ne 0$ and $H^i(G)=0$ for $i>n$. The cohomological dimension of $G$ is said to be infinite if no such $n$ exists. The cohomological dimension of $G$ is denoted by $\cd (G)$. The cohomological dimension of a free pro-$p$-group is 1. In ~\cite{Lab67} we show that If $\dim H^2(G)=1$ and the cup-product
 $$ H^1(G)\otimes H^1(G)\rightarrow H^2(G)$$
 is a non-degenerate bilinear form then $\cd(G)=2$. This is the case when $G$ is the Galois group of the maximal $p$-extension of  $\Q_p(\zeta_p)$, where $\Q_p$ is the $p$-adic number field and $\zeta_p$ is a primitive $p$-th root of unity.

 We now unravel the statement of Theorem~\ref{B} in the context of a presentation of a pro-$p$-group $G$. If $F=F(x_1,\ldots,\_d)$ is the free pro-$p$-group on $x_1,....,x_d$ and
 $$1\rightarrow R\rightarrow F\rightarrow G\rightarrow 1$$
 is a presentation of $G$ we have an exact sequence
$$0\rightarrow H^1(G)\rightarrow H^1(F)\rightarrow H^1(R)^F\rightarrow H^2(G)\rightarrow H^2(F),$$
where $H^1(G)$ is the dual of $G/G^p[G,G]$. Here $[G,G]$ is the closed subgroup of $G$ generated by the commutators $[x,y]=x^{-1}y^{-1}xy$ with $x,y$ in $G$.
This implies that $\dim H^1(G)$ is the minimal number of generators of $G$. If $\dim H^1(F)=\dim H^1(G)$ the presentation is called minimal, in which case $H^2(G)=H^1(R)^F=H^1(R/R^p[R,F])$. We let $\xi_i$ be the image of $x_i$ in $F/F^p[F,F]=G/G^p[G,G]$ and $\chi_1,\ldots,\chi_d \in H^1(G)=H^1(F)$ the basis of $H^1(F)=H^1(G)$ dual to $\xi_1,\ldots,\xi_d$. Thus, for a minimal presentation, $\dim H^2(G)$ is the minimal number of elements of $R$ that generate $R$ as a closed normal subgroup of $R$.  The presentation is minimal if and only if $R\le F^p[F,F]$. 

Let $r_1,\ldots,r_m$ be elements of $F^p[F,F]$, let $R=(r_1\ldots,r_m)$ be the closed normal subgroup of $F$ generated by $r_1,\ldots,r_m$ and let $G=F/R$.  We have
$$r_k=\prod_{j=1}^d x_j^{pc_{kj}}\prod_{1\le i<j\le d} [x_i,x_j]^{a_{ijk}}s_k$$
with $s_k\in [F,[F,F]]$ and $c_{kj}, a_{ijk}\in\Z_p$. Let $\overline r_k$ be the image of $r_k$ in $R/R^p[R,F]=H^2(G)^*$, the dual of $H^2(G)$. Then $\overline r_1,\ldots,\overline r_m$ generate $H^2(G)^*$ so that $m\ge \dim H^2(G)$ which is $\ge1$ if $G$ is not a free pro-$p$-group.  An important basic fact is that $\overline r_k(\chi_i\cup\chi_j)=\overline a_{ijk}$, the image of $a_{ijk}$ in $\F_p=\Z_p/p\Z_p$.
Dualizing the cup product mapping 
$$\phi:H^1(G)\otimes H^1(G)\rightarrow H^2(G),$$
we get, when $p\ne2$, a mapping
 $$\phi^*:H^2(G)^*\rightarrow\bigwedge^2W=\bigoplus _{1\le i<j\le d}\F_p\xi_i\wedge\xi_j$$
where $W=F/F^p[F,F]$. We have $\langle\chi_i\cup\chi_j,\overline r_k\rangle=\rho_k(\chi_i\cup\chi_j)=\overline a_{ijk}$, But $\chi_i\cup\chi_j=\phi(\chi_i\otimes\chi_j)$ so that $\overline a_{ijk}=\langle\chi_i\otimes\chi_j,\phi^*(\overline r_k)\rangle$ which implies that
$$\phi^*(\overline r_k)=\sum_{1\le 1<j\le d}\overline a_{ijk}\xi_i\wedge\xi_j.$$
If $H^1(G)$ is the direct sum of non-trivial subspaces $U,V$ we have, after a possible change of basis,
 $$U=\bigoplus_{i\in A}\F_p\chi_i\  \  V=\bigoplus_{i \in B}\F_p\chi_i,$$
 where $A,B$ is a partition of  $\{1,\ldots,d\}$.
 If $\psi$ is the restriction of $\phi$ to $U\otimes V$, then $ \psi$ maps $U\otimes V$ surjectiively onto $H^2(G)$  if and only if
 $$\psi^*:H^2(G)^*\rightarrow Z=\bigoplus_{ i\in A ,\  j \in B} \F_p\xi_i\wedge\xi_j.$$
 is injective; note that 
 $$\psi^*(\overline r_k)=\sum_{ i\in A,\  j\in B}a_{ijk}\xi_i\wedge\xi_j.$$
 If $\sigma$ is the projection of $\bigwedge^2W$ onto $Z$ we have $\psi^*=\sigma\phi^*$  so that $\psi^*$ is injective if and only if $\phi^*$ is injective, the latter being  equivalent to $m$ being equal to the rank of the matrix $M_L$ whose columns are indexed by the pairs $(i,j)$ with $i\in A,j\in B$ and whose entry in the $k$-th row and $(i,j)$-th column is $\overline a_{ijk}$. Thus, to prove Theorem~\ref{B}, we have to prove that $\cd(G)=2$ if the following two conditions 
 \begin{enumerate}[\ \rm (a)]
 \item $\overline a_{ijk}=0$ when $i,j\in A$, 
  \item the rank of $M_L$ is $m$,
  \end{enumerate}
 hold for the elements 
 $$\rho_k=\sum_{1\le 1<j\le d}\overline a_{ijk}\xi_i\wedge\xi_j.$$ 
 Note that that ${\rm rank }\  M_L=m$ implies $\dim H^2(G)=m$.
 If $ L=L(X)$ is the free $\F_p$-Lie algebra on $X=\{\xi_1,\ldots,\xi_d\}$ we can identify $\bigwedge^2 W$ the space of degree $2$ elements of $L$ with $\xi_i\wedge\xi_j$ correspnding to $[\xi_i,\xi_j].$ So that 
 $$
 \rho_k=\sum_{1\le i<j\le d}\overline a_{ijk}[\xi_i,\xi_j].
 $$
 But, by our criterion for strong freeness, the sequence $\rho_1,\ldots,\rho_m$ is strongly free and hence $\cd(G)=2$.

\section{Epilog}
In \cite{Sch} Alexander Schmidt extended our results to the case of global fields. In ~\cite{LM2011} our results for $p$ odd were extended to $p=2$. In \cite{Forre2011} Forr\'e gave an independent treatment of our results which covered the case $p=2$. In ~\cite{Gar} Gartner, using higher Masssey products, extended our cohomological criterion for mildness to obtain mild pro-$p$-groups defined by relators of arbitrary degree.


\begin{thebibliography}{99}
\bibitem{Lyndon} R. Lyndon, {\it Cohomology theory of groups with a single defining relation}, Annals of Math.,Series 2,t.52, (1950), 650-665.
\bibitem{Serre63}Jean-Pierre Serre,{\it Structure de certains pro-p-groupes},S\'eminaire Bourbaki, 252, (1963), 1-11.
\bibitem{Bru66} A. Brumer, {\it Pseudocompact algebras, profinite groups and class formations}, Journal of Algebra, 4, (1966), 442-470.
\bibitem{Lab67}J. P. Labute, {Alg\`ebres de Lie et pro-p-groupes d\'efinis par une seule relation}, Invent. Math. 4,(1967), 142-158.
\bibitem{Wald67}H. V. Waldinger, {\it On extending Witt's formula}, J. Algebra 5, 41-58 (1967),
\bibitem{Gild68} D. Gildenhuys, {\it On pro-p-groups with a single defining relation}, Invent. Math. 5 (1968), 357-366.
\bibitem{Lab70} J. P. Labute, {\it On the descending central series of groups with a single defining relation}, J. Algebra 14, 16-23 (1970).
\bibitem{Kuz69} L.K. Kuzmin, {\it Homology of profinite groups, Schur multipliers, and class field theory}, Izv. Akad. Nauk SSSR 33 (1969),1149-1181.
\bibitem{Koch70}H. Koch, {Galoissche Theorie der $p$-Erweiterungen},Berlin(1970).
\bibitem{Koch77}H. Koch, {\it \"Uber pro-$p$-gruppen der kohomologischen dimension $2$}, Math. Nachr. 78 (1977), 285-289.
\bibitem{Anick82}D. Anick, {\it Non-commutative algebras and their Hilbert series}, J. Algebra 78 (1982), 120–140.
\bibitem{Lab85}J. Labute, {\it The Determination of the Lie
algebra associated to the  central series of a group}, Trans.
Amer. Math. Soc. 288 (1985), 51-57.
\bibitem{Anick87} D. Anick, {\it Inert sets and the Lie algebra associated to a group}, 111 (1987), 154-165.
\bibitem{Lab2006}J. Labute, {\it Mild pro-$p$-groups and Galois groups of
$p$-extensions of $\Q$}, J. Reine Angew. Math. 596 (2006), 115-130.
\bibitem{Sch}A. Schmidt,{\it \"Uber pro-$p$-fundamentalgruppen markierte  arithmetischer curven}, J. Reine Angew. Math., 640 (2010), 203-235..
\bibitem{Forre2011}P. Forr\'e,{\it Strongly free sequences and pro-$p$-groups of cohomological dimension $2$}, J. Reine Angew. Math. 658,(2011).
\bibitem{LM2011}J. Labute, J. Min\'a\u c, {\it Mild pro-$2$-groups and $2$-extensions of $\Q$ with restricted ramification}, J. Alg, 332 (2011), 136-158
\bibitem{Gar}J. G\"artner, {Higher Massey products in the cohomology of mild of mild pro-$p$-groups}, 422 (2015)788-822.

\end{thebibliography}
\end{document}